\documentclass{amsart}
\usepackage{graphicx}
\newtheorem{thm}{Theorem}[section]
\newtheorem{cor}[thm]{Corollary}

\theoremstyle{definition}

\theoremstyle{remark}

\numberwithin{equation}{section}

\begin{document}

\title[]{On functional inequalities associated with Drygas functional equation}%
\author{Manar Youssef and  Elqorachi Elhoucien}%
\address{Manar Youssef, University
Ibn Zohr, Superior School of Technology, Guelmim, Morocco}%
\email{manaryoussef1984@gmail.com}%
\address{Elqorachi Elhoucien, University
Ibn Zohr, Department of Mathematics, Faculty of Sciences, Agadir,
Morocco} \email{elqorachi@hotmail.com}

\begin{abstract}
In the paper, the equivalence of the functional inequality
$$\|2f(x)+f(y)+f(-y)-f(x-y)\|\leq\|f(x+y)\|\;\;\;(x,y\in{G})$$
and the Drygas functional equation
$$f(x+y)+f(x-y)=2f(x)+f(y)+f(-y)\;\;\;(x,y\in{G})$$
is proved for functions $f:G\rightarrow E$ where $(G, +)$ is an
abelian group, $(E, <\cdot, \cdot>)$ is an inner product space, and
the norm is derived from the inner product in the usual way.
\end{abstract}
\maketitle
\section{Introduction}
Throughout the paper, $(G, +)$ will denote an abelian group and $(E,
<\cdot, \cdot>)$ an inner product space over $\mathbb{K}$
($\mathbb{R}$ or $\mathbb{C}$) with inner product $<\cdot, \cdot>$
and associated norm $\|\cdot\|$.\\
Gy. Maksa and P. Volkman proved in \cite{10b} the following
\begin{thm}\label{thm01} Let $G$ be a group, $E$ be an inner product
space. If $f:G\rightarrow E$ be a function such that
$$\|f(x)+f(y)\|\leq\|f(x y)\|$$
for all $x,y\in G$. Then $f$ satisfies
$$ f(x y)=f(x)+f(y)$$
for all $x,y\in G$.
\end{thm}
 In \cite{5a}, A. Gil\'{a}nyi showed that if  $G$ is a $2$-divisible abelian group,
then the functional inequality
\begin{equation}
\|2f(x)+2f(y)-f(x-y)\|\leq \|f(x+y)\|\;\;\text{for all}\;x,y\in G
\end{equation}
 implies
\begin{equation}\label{eqx}
f(x+y)+f(x-y)=2f(x)+2f(y)\;\;\text{for all}\;x,y\in G,
\end{equation}
and the  commutativity of $G$ may be replaced by the Kannappan
condition\\
 $f(xyz)=f(xzy)$, $x,y,z\in G$.\\
 In \cite{ratz} J. R\"{a}tz deleted the 2-divisibility of $G$, weakened the Kannappan condition  and
discussed variants of Gil\'{a}nyi result. \\
In \cite{Springer}, E. Elqorachi et al. proved that if the function
$f:G\rightarrow E$ from $G$ (an abelian 2-divisible group) to an
inner product space $E$, satisfies the inequality
$$\|2f(x)+2f(y)-f(x+\sigma(y))\|\leq\|f(x+y)\|,
 \;\;x,y\in{G},$$
  where $\sigma:G\rightarrow G$ is an involution (i.e., $\sigma(x+y)=\sigma(x)+\sigma(y)$,
  $\sigma\circ\sigma(x)=\sigma(x)$ for all $x,y\in{G}$),
  then $f$ satisfies the $\sigma$-quadratic functional equation
 $$f(x+ y)+f(x+\sigma(y))=2f(x)+2f(y),\;x,y\in{G}.$$
The above-described effect: inequality implies equality was proved
for some other functional equations. The interested reader can refer
to \cite{An}, \cite{Cho}, \cite{Chung}, \cite{Ebadian},
 \cite{Kim1}, \cite{Kim2}, \cite{Kwon}, \cite{Park1}, \cite{Park2}, \cite{C. Park1}, \cite{C. Park2},
  \cite{C. Park3}, \cite{Roh} and \cite{Volk} for a through account on the subject of functional inequalities. \\
We say that the function $f:G\rightarrow E$ satisfies the Drygas
functional equation, if
\begin{equation}\label{eqDrygas1}
f(x+y)+f(x-y)=2f(x)+f(y)+f(-y)
\end{equation}
for all $x,y\in{G}$.\\
 The equation was introduced in \cite{Dryg},
where the author was looking for characterizations of quasi inner
product spaces, which in turn led to solutions of some problems in
statistics and
mathematical programming.\\
The functional equation (\ref{eqDrygas1}) has been studied by Gy.
Szabo \cite{Szabo}, B. R. Ebanks et al. \cite{Ebanks}, V. A.
Fa\u{i}ziev and P. K. Sahoo \cite{Fa1}. The solutions of equation
(\ref{eqDrygas1}) in abelian group are obtained by H. Stetk{\ae}r
in \cite{St}.\\\\
The purpose of our paper is to show that if $f:G\rightarrow E$
satisfies the Drygas inequality
$$\|2f(x)+f(y)+f(-y)-f(x-y)\|\leq \|f(x+y)\|\;\; \text{for all}\;x,y\in G,$$
 then $f$ satisfies the Drygas functional equation (\ref{eqDrygas1}).\\
 Throughout this paper, $f^{o}$ and $f^{e}$ denote the
odd and even parts of $f$, respectively, i.e.,
$f^{o}(x)=\displaystyle\frac{f(x)-f(-x)}{2}$,
$f^{e}(x)=\displaystyle\frac{f(x)+f(-x)}{2}$ for all $x\in{G}$.
\section{Main result}
\begin{thm}\label{thm21}
Let $G$ be an abelian group, $E$ be an inner product space and
 $f:G\rightarrow E$ be a mapping such that
\begin{equation}\label{eq21}
\|2f(x)+f(y)+f(-y)-f(x-y))\|\leq \|f(x+y)\|
\end{equation}
for all $x,y\in{G}$. Then $f$ is a solution of the Drygas functional
equation
\begin{equation}\label{eqD}
f(x+y)+f(x-y)=2f(x)+f(y)+f(-y),\;\;\;\;x,y\in{G}.
\end{equation}
\end{thm}
\begin{proof}In the proof we use Gy. Maksa and Volkmann's \cite{10b}
and Gil\'{a}nyi \cite{5a} results to prove that $f^{e}$ is a
solution of the quadratic functional equation (\ref{eqx}) and
$f^{o}$ is a solution of the Cauchy functional equation.\\
Writing $x=y=0$ in (\ref{eq21}), we  obtain $3\|f(0)\|\leq \|f(0)\|$, so $f(0)=0$.\\
Replacing $y$ by $-x$ in (\ref{eq21}), we get
\begin{equation}\label{eq23}
2f(x)+2f^{e}(x)=f(2x).
\end{equation}
By using $f=f^{e}+f^{o}$ and (\ref{eq23}), we have
\begin{equation}\label{eq24}
4f^{e}(x)+2f^{o}(x)=f^{e}(2x)+f^{o}(2x).
\end{equation}
If we replace $x$ by $-x$ in (\ref{eq24}), we get
\begin{equation}\label{eq25}
4f^{e}(x)-2f^{o}(x)=f^{e}(2x)-f^{o}(2x).
\end{equation}
By adding and subtracting (\ref{eq24}) to (\ref{eq25}), we obtain
respectively,
\begin{equation}\label{eq26}
f^{e}(2x)=4f^{e}(x)
\end{equation}
and
\begin{equation}\label{eq27}
f^{o}(2x)=2f^{o}(x).
\end{equation}
By using (\ref{eq26}) and (\ref{eq27}), we can easy to check by
induction that
\begin{equation}\label{eq28}
f^{e}(x)=4^{-n}f^{e}(2^{n}x)\;\;\text{and}\;\;
 f^{o}(x)=2^{-n}f^{o}(2^{n}x)
 \end{equation}
for all $n\in \mathbb{N}$ and $x\in G$.\\
Substituting $x$ by $-x$ and $y$ by $-y$  in (\ref{eq21}), we have
\begin{equation}\label{eq29}
\|2f(-x)+2f^{e}(y)-f(-x+y))\|\leq \|f(-x-y)\|,\;\;x,y\in{G}.
\end{equation}
 By adding (\ref{eq29}) to (\ref{eq21}) and using the triangle
 inequality, we obtain
\begin{equation}\label{eq210}
\|4f^{e}(x)+4f^{e}(y)-2f^{e}(x-y)\|\leq \|f(x+y)\|+\|f(-x-y)\|.
\end{equation}
Writing $2^{n}x$ instead of $x$ and $2^{n}y$ instead of $y$ in
(\ref{eq210}), we get
\begin{align}\label{eq211}
\nonumber \|4f^{e}(2^{n}x)+4f^{e}(2^{n}y)-2f^{e}(2^{n}(x-y))\|\leq&
\|f^{e}(2^{n}(x+y))+f^{o}(2^{n}(x+y))\|\\
 &+\|f^{e}(2^{n}(x+y))+f^{o}(2^{n}(-x-y))\|.
\end{align}
By using the induction assumption (\ref{eq28}) and  dividing the new
inequality by $4^{n}$, we have
\begin{align}\label{eq212}
\nonumber \|4f^{e}(x)+4f^{e}(y)-2f^{e}(x-y)\|\leq&
\|f^{e}(x+y)+2^{-n}f^{o}(x-y)\|\\
&+\|f^{e}(x+y)+2^{-n}f^{o}(-x-y)\|.
\end{align}
By letting $n\rightarrow+\infty$ in the last inequality, we obtain
\begin{equation}\label{eq213}
\|2f^{e}(x)+2f^{e}(y)-f^{e}(x-y)\|\leq \|f^{e}(x+y)\|,\;\;x,y\in{G}.
\end{equation}
It was proved in \cite{5a,ratz} that this inequality is equivalent
to the quadratic functional equation
\begin{equation}\label{eq214}
f^{e}(x+y)+f^{e}(x-y)=2f^{e}(x)+2f^{e}(y),\;\;x,y\in{G}.
\end{equation}
Which proves the first part of our statement.
 From (\ref{eq21}) we have
\begin{equation}\label{eq215}
\|2f^{e}(x)+2f^{o}(x)+ 2f^{e}(y)-f^{e}(x-y)-f^{o}(x-y)\|\leq
\|f(x+y)\|.
\end{equation}
Since $f^{e}$ satisfies the quadratic functional equation
(\ref{eq214}), we get
\begin{equation}\label{eq216}
\|f^{e}(x+y)+2f^{o}(x)-f^{o}(x-y)\|\leq \|f(x+y)\|.
\end{equation}
Interchanging the roles of $x$ and $y$ in (\ref{eq216}) we obtain
\begin{equation}\label{eq217}
\|f^{e}(y+x)+2f^{o}(y)-f^{o}(y-x)\|\leq \|f(y+x)\|.
\end{equation}
Adding this inequality  to (\ref{eq216}), we get
\begin{equation}\label{eq218}
\|f^{e}(x+y)+f^{o}(x)+f^{o}(y)\|\leq \|f^{e}(x+y)+f^{o}(x+y)\|\;\;
\text{for all}\;x,y\in{G}.
\end{equation}
Inequality (\ref{eq218}) can be rewritten as follows
\begin{align*}
 \|f^{o}(x)+f^{o}(y)\|^{2}&+\|f^{e}(x+y)\|^{2}+2\text{Re}\langle
 f^{o}(x)+f^{o}(y),f^{e}(x+y)\rangle\\
&\leq \|f^{e}(x+y)\|^{2}+\|f^{o}(x+y)\|^{2}+2\text{Re}\langle
f^{o}(x+y),f^{e}(x+y)\rangle,\;\;\text{so,}
 \end{align*}
\begin{equation}\label{eq219}
\|f^{o}(x)+f^{o}(y)\|^{2}+2\text{Re}\langle
 f^{o}(x)+f^{o}(y)-f^{o}(x+y),f^{e}(x+y)\rangle\leq
 \|f^{o}(x+y)\|^{2}.
\end{equation}
In (\ref{eq219}), write  $-x$ and $-y$ instead of $x$ and $y$,
respectively  and add the inequality so obtained to (\ref{eq219}) to
obtain
\begin{equation}\label{eq220}
\|f^{o}(x)+f^{o}(y)\|\leq\|f^{o}(x+y)\|,\;\;x,y\in G.
\end{equation}
In view of \cite{10b}, the inequality (\ref{eq220}) is equivalent to
the Cauchy functional equation
\begin{equation}\label{eq221}
f^{o}(x+y)=f^{o}(x)+f^{o}(y),\;\;x,y\in G.
\end{equation}
 Thus, since
$f(x)=f^{e}(x)+f^{o}(x)$, we can easily check that $f$ is a solution
of Drygas functional equation  (\ref{eqD}). This completes the
proof. \phantom{========================}$\spadesuit$
\end{proof}

 The commutativity of $G$ used in
Theorem \ref{thm21} may be replaced by the Kannappan condition:
$f(xyz)=f(yxz)$, for all $x,y,z\in G$.
\begin{cor} If $G$ is group (not necessarily abelian) and $E$ an inner product space. Then the Drygas inequality
$$\|2f(x)+f(y)+f(y^{-1})-f(xy^{-1})\|\leq \|f(xy)\|$$
with $f(xyz)=f(yxz)$ for all $x,y,z\in G$, is equivalent to Drygas
functional equation
$$f(xy)+f(xy^{-1})=2f(x)+f(y)+f(y^{-1})\;\;\;\text{for all}\; x,y\in G.$$
\end{cor}


\begin{thebibliography}{20}
\bibitem{An} {\sc J. Su An}, {\it On functional inequalities associated with
Jordan-von Neumann type functional equations}, Commun. Korean Math.
Soc., \textbf{23} 3 (2008), 371-376.

\bibitem{Cho} {\sc Y.-S. Cho and K.-W. Jun}, {\it The stability of functional
inequalities with additive mappings}, Bull. Korean Math. Soc.,
\textbf{46} 1 (2009), 11-23.

\bibitem{Chung} {\sc S.-C. Chung, S.-B. Lee and W.-G. Park}, {\it On the stability
of an additive functional inequality}, Int. Journal of Math. Anal.,
\textbf{6} 53 (2012), 2647-2651.

\bibitem{Ebadian} {\sc A. Ebadian, N. Ghobadipour, Th. M. Rassias and M.
E. Gorjdi}, {\it Functional inequalities associated with Cauchy
additive functional equation in non-Archimedean spaces}, Disc. Dyn.
Nat. Soc., \textbf{2011}, Article ID 929824, 14 pages.



\bibitem{Dryg} {\sc H. Drygas}, {\it Quasi-inner products and their applications}, in: A.K.
Gupta (Ed.), Advances in Multivariate Statistical Analysis, D.
Reidel Publishing Co., 1987, pp. 13–30.

\bibitem{Ebanks} {\sc B. R. Ebanks, Pl. Kannappan and P. K. Sahoo}, {\it A common
generalization of functional equations characterizing normed and
quasi-inner product spaces}, Canad. Math. Bull.,
 \textbf{35} (3) (1992), 321-327.

\bibitem{Springer} {\sc E. Elqorachi, Y. Manar and Th. M. Rassias},
{\it Hyers-Ulam stability of the quadratic functional equation}.
Functional Equations in Mathematical Analysis, Springer optimization
and its applications, V\textbf{52} (2012), 97-105.

 \bibitem{Fa1} {\sc V. A. Fa\v{i}ziev and P. K. Sahoo}, {\it On Drygas functional equation on
groups}, Int. J. Appl. Math. Stat., \textbf{7} (2007), 59-69.


\bibitem{5a} {\sc A. Gil\'{a}nyi}, {\it Eine zur Parallelogrammgleichung
aquivalente Ungleichung}, Aequationes Math., \textbf{62} (2001),
303-309.

\bibitem{Kim1} {\sc H.-M. Kim, J. Lee and E. Son}, {\it Approximate functional
inequalities by additive mappings}, J. Math. Ineq., \textbf{6} 3
(2012), 461-471.

\bibitem{Kim2} {\sc H.-M. Kim, S.-Y. Kang and I.-S. Chang}, {\it On functional
inequalities originating from module Jordan left derivations}, J.
Ineq. Appl., \textbf{2008}, Article ID 278505, 9 pages.


\bibitem{Kwon} {\sc Y. H. Kwon, H. M. Lee, J. S. Sim, J. Yang and C. Park},
{\it Jordan-von Neumann type functional inequalities}, J.
Chungcheong Math. Soc., \textbf{20} 3, September 2007.



\bibitem{10b} {\sc Gy. Maksa and P. Volkmann}, {\it Caracterization of group
homomorphisms having values in an inner product space}, Publ. Math.,
\textbf{56} (2000), 197-200.

\bibitem{Park1} {\sc W.-G. Park}, {\it Hyers-Ulam stability of an additive
functional inequality}, Int. Journal of Math. Anal., \textbf{6} 14
(2012), 681-686.

\bibitem{Park2} {\sc W.-G. Park and M. H. Han}, {\it Stability of an additive
functional inequality with the fixed point alternative}, Int. J.
Pure Appl. Math., \textbf{77} 3 (2012), 403-411.


\bibitem{C. Park1} {\sc C. Park, J. Su An and F. Moradlou}, {\it Additive functional
inequalities in Banach modules}, J. Ineq. Appl., \textbf{2008},
Article ID 592504, 10 pages.

\bibitem{C. Park2} {\sc C. Park and J. R. Lee}, {\it Comment on "Functional inequalities
associated with Jordan-von Neumann type additive functional
equations"}, J. Ineq. Appl., \textbf{2012}, 9 pages.


\bibitem{C. Park3} {\sc C. Park, Y. S. Cho and M.-H. Han}, {\it Functional
inequalities associated with Jordan-von Neumann-type additive
functional equations}, J. Ineq. Appl., \textbf{2007}, Article ID
41820, 13 pages.


\bibitem{ratz} {\sc J. R\"{a}tz}, {\it On inequality associated with the
Jordan-von Neumann functional equation}, Aequationes Math.,
\textbf{54} (2003), 191-200.

\bibitem{Roh} {\sc J. Roh and I.-S. Chang}, {\it Functional inequalities
associated with additive mappings}, Abs. Appl. Anal., \textbf{2008},
Article ID 136592, 11 pages.


\bibitem{St} {\sc H. Stetk\ae r},
{\it Functional equations on abelian groups with involution, II},
 Aequationes Math., V\textbf{55} (1998), 227-240.

\bibitem{Szabo} {\sc Gy. Szabo},
 {\it Some functional equations related to
quadratic functions}, Glasnik Math., \textbf{38} (1983), 107-118.

\bibitem{Volk} {\sc P. Volkmann},
 {\it Pour une fonction r\'{e}elle $f$
l'in\'{e}quation $|f(x)+f(y)|\leq |f(x+y)|$ et l'\'{e}quation de
Cauchy sont \'{e}quivalentes}, Proc. of the Twenty-third
International Symposium on Functional Equations (Gargnano, Italy,
1985), Centre for Information Theory, Faculty of Mathematics,
University of Waterloo, Waterloo, Ontario, Canada, 43.

\end{thebibliography}
\end{document}